\documentclass[12pt]{amsart}
\usepackage[utf8]{inputenc}
\usepackage[margin=1.5in]{geometry}
\usepackage[lite]{amsrefs}
\usepackage{amsfonts}
\usepackage{mathrsfs}
\usepackage{amscd}

\usepackage{physics}
\usepackage{amsmath}
\usepackage{mathtools}
\usepackage{amssymb}
\usepackage{multicol}
\usepackage{graphicx}
\usepackage{wrapfig}
\usepackage{float}
\usepackage[english]{babel}
\usepackage[utf8]{inputenc}
\usepackage{fancyhdr}
\usepackage{amsthm}
\usepackage[all]{xy}
\usepackage{tikz-cd}
\usepackage{bbold}
\usepackage{hyperref}
\usepackage{tabularx,booktabs,caption,ragged2e}
\newcolumntype{L}{>{\RaggedRight\arraybackslash}X}
\newtheorem*{ThmA}{Main Theorem}

\newtheorem{thm}{Theorem}[section]

\newtheorem{lemma}[thm]{Lemma}
\newtheorem{propn}[thm]{Proposition}

\theoremstyle{definition}
\newtheorem{defn}[thm]{Definition}
\newtheorem{example}[thm]{Example}

\theoremstyle{remark}
\newtheorem{remark}[thm]{Remark}

\newcommand{\R}{\mathbb{R}}

\newcommand{\Z}{\mathbb{Z}}

\newcommand{\g}{\mathfrak{g}}

\newcommand{\Ad}{\operatorname{Ad}}
\newcommand{\ad}{\operatorname{ad}}

\usepackage{graphicx} 
\allowdisplaybreaks 

\title{An Explicit Construction of $\mathbb{S}^1$-Gerbes over the Stack $[G/G]$}
\author{Dadi Ni}
\address{School of Mathematics and Statistics, Henan University, China} 
\email{\href{mailto:nidd@henu.edu.cn}{nidd@henu.edu.cn}}
\author{Kaichuan Qi}
\address{Department of Mathematics, Penn State University, USA} 
\email{\href{mailto:kaichuan@psu.edu}{kaichuan@psu.edu}}

\thanks{The first author is  supported by the Natural Science Foundation of Henan Province
 (No. 252300421766). The second author's research is partially supported by the National Science Foundation (award DMS-2302447).}

\begin{document}

\maketitle

\begin{abstract}
For a compact and connected Lie group $G$, we present an explicit construction of an $\mathbb{S}^1$-gerbe over the differentiable stack $[G/G]$ in the framework of $\mathbb{S}^1$-central extensions of Lie groupoids. This gives a complete proof of the construction outlined earlier by Behrend--Xu--Zhang, together with an explicit proof of the differential-form identity stated there without proof. In particular, when $G$ is compact, simple, and simply connected, the Dixmier--Douady class of the resulting gerbe is the canonical generator of ${\rm H}^3_G(G,\mathbb Z)$.
\end{abstract}

\begin{itemize}
	\item 
	\item 
\end{itemize}

\tableofcontents

\section{Introduction}
Gerbes are higher geometric analogues of line bundles, designed to encode
topological information in degree three. More precisely, just as complex line
bundles on a manifold $M$ are classified by classes in
${\rm H}^2(M,\mathbb Z)$, $\mathbb S^1$-gerbes give differential geometric
representatives for classes in ${\rm H}^3(M,\mathbb Z)$, their fundamental
invariant being the Dixmier--Douady class. Gerbes were originally introduced
by Giraud \cite{Giraud} as certain sheaves of groupoids, and later became
central in differential geometry and mathematical physics through the work of
Hitchin \cites{Hit01,Hitchin_BundleGerbes} and Brylinski \cite{Brylinski}.
Among concrete geometric models, bundle gerbes, introduced by Murray
\cite{Murray_BundleGerbes}, provide a particularly useful differential
geometric realization of degree-three integral cohomology classes. We refer to
\cite{Bunk1} for recent developments and applications of bundle gerbes in
geometry and physics.

Gerbes also occur naturally in equivariant geometry, Wess--Zumino--Witten
models, string theory, and twisted $K$-theory
\cites{Gawedzki-Reis,Meinrenken,Bunk2}. Related constructions have also been
studied for orbifolds \cite{MR2045679}. Motivated in part by these examples,
Behrend and Xu developed a framework for $\mathbb S^1$-gerbes over
differentiable stacks, identifying them with $\mathbb S^1$-central extensions
of Lie groupoids up to Morita equivalence \cite{stacks_B-X}. This point of
view is especially well suited to quotient stacks and equivariant situations,
where stack-theoretic cohomology retains the symmetry information lost by
passing to an ordinary quotient space.

A particularly important class of examples arises from the differentiable
stack $[G/G]$, where $G$ acts on itself by conjugation. Geometrically,
$[G/G]$ may be regarded as a refined quotient which encodes the conjugation
symmetry of $G$ in a way that remains well behaved even when the ordinary
quotient space is singular. In particular, the cohomology of $[G/G]$ is
naturally identified with the equivariant cohomology
${\rm H}^{\bullet}_G(G,\mathbb Z)$.
Within this framework, the Cartan 3-form on 
$G$, constructed from an invariant bilinear form on the Lie algebra, defines a canonical equivariant cohomology class, which serves as a fundamental example of a degree-three invariant associated to 
$G$. A central goal is to realize this canonical equivariant degree-three class as the Dixmier–Douady class of an 
$\mathbb{S}^1$-gerbe over $[G/G]$, commonly referred to as the basic gerbe.

The problem of constructing the basic equivariant $\mathbb{S}^1$-gerbe over $G$ for the conjugation action has been studied from several viewpoints. One classical line of approach uses local Čech-type or bundle gerbe models. In this direction, Gawedzki and Reis constructed gerbes associated to the WZW model in the special case of $G = SU(N)$ \cite{Gawedzki-Reis}. More generally, Meinrenken \cite{Meinrenken}  constructed the basic
equivariant gerbe for any compact, simple, simply connected Lie group 
$G$, and see \cite{Gawedzki-Reis-nonsc} for further development.

A different, global framework is developed by Behrend and Xu, who showed that $\mathbb{S}^1$-gerbes over differentiable stacks can be described as $\mathbb{S}^1$-central extensions of Lie groupoids, up to Morita equivalence \cite{stacks_B-X}. Under this framework, Behrend, Xu, and Zhang (BXZ) outlined a construction of the basic gerbe over $[G/G]$ by groupoid methods \cite{B-X-Z}. On the other hand, Xu's work on quasi-symplectic groupoids and Morita equivalence provides the groupoid-theoretic background for the comparison between loop group and group valued moment map pictures used in this paper. Our goal here is to return to the BXZ construction and supply a complete proof in the groupoid-central-extension framework, together with an explicit geometric realization and  comparison at the level of differential forms.

The fundamental topological invariant of an $\mathbb{S}^1$-gerbe is its Dixmier--Douady class, a degree-three integral cohomology class on the underlying differentiable stack, which may be described using suitable connection and curving data \cites{stacks_B-X,Murray_BundleGerbes,Hitchin_BundleGerbes}. In this paper, we construct an $\mathbb{S}^1$-gerbe over 
$[G/G]$ whose Dixmier--Douady class coincides with the canonical equivariant class in ${\rm H}^3_G(G,\mathbb{Z})$.

\begin{ThmA}
Let $G$ be a compact, simple, and simply connected Lie group, and let $\alpha \in {\rm H}^3_G(G,\mathbb{Z})$ be the canonical generator. There exists an 
$\mathbb{S}^1$-central extension of Lie groupoids representing an 
$\mathbb{S}^1$-gerbe over the differentiable stack 
$[G/G]$, whose Dixmier-Douady class is $\alpha$.
\end{ThmA}

The assumptions on $G$ are used at different points in the argument. The differential-form comparison of Section 4 is established for compact connected $G$. The passage through the Kac–Moody $\mathbb{S}^1$-central extension of $LG$, and hence the construction of the central extension of Lie groupoids used in Section 5, requires the simply connected hypothesis in the present setting. The stronger assumption that $G$ be simple is used only in the final identification of the Dixmier–Douady class with the canonical generator of ${\rm H}^3_G(G,\mathbb Z)$.

The proof is based on symplectic reduction and the theory of symplectic groupoids. Starting from the Kac--Moody $\mathbb S^1$-central extension
$\widetilde{LG}\to LG$, we consider the trivialized cotangent bundle of
$\widetilde{LG}$, which naturally carries both a symplectic groupoid structure and a Hamiltonian
$\mathbb S^1$-action. Performing symplectic reduction at the regular value $1$ of the moment map yields a reduced symplectic groupoid with underlying space
$LG\times L\mathfrak g^*$. The corresponding level set, together with the reduced groupoid, gives rise to the desired $\mathbb S^1$-central extension of Lie groupoids, while the reduced symplectic form provides the connection and curving data for the associated gerbe.

To identify the resulting Dixmier--Douady class, we compare the reduced symplectic form on
$LG\times L\mathfrak g^*\rightrightarrows L\mathfrak g^*$ with the Alekseev--Malkin--Meinrenken (AMM) cocycle on the conjugation action groupoid
$G\times G\rightrightarrows G$. At the level of Hamiltonian spaces, this comparison goes back to the quasi-Hamiltonian framework of AMM, while Xu formulated the corresponding quasi-symplectic groupoid viewpoint. For connected Lie groups, Alekseev--Meinrenken's path-fibration and Atiyah-algebroid approach provides an invariant primitive underlying this comparison. In the present paper, we give a direct and explicit proof of the required form-level identity in the concrete model
$LG\times L\mathfrak g^*\rightrightarrows L\mathfrak g^*$. Under the assumption that $G$ is compact, simple, and simply connected, the AMM cocycle represents the canonical generator of ${\rm H}^3_G(G,\mathbb Z)$, and hence the Dixmier--Douady class of the gerbe constructed above coincides with $\alpha$.



\section{Preliminaries}\label{Sec:preliminary}
This section is an exposition of basic definitions and examples related to symplectic groupoids and the de Rham cohomology of Lie groupoids, following the conventions outlined in \cites{crainic2021lectures, mackenzie2005general}.

A groupoid is a small category in which every morphism admits an inverse. Concretely, it consists of two sets, $\mathcal{G}$ (the arrows) and $M$ (the objects), together with the following structure maps satisfying group-like axioms:
\begin{enumerate}
    \item Source and target maps 
    $\mathbf{s}, \mathbf{t}: \mathcal{G} \to M$, assigning to each arrow its domain and codomain;
    \item Unit map 
    $\epsilon: M \to \mathcal{G}$, embedding objects as identity arrows;
    \item Inverse map 
    $\iota: \mathcal{G} \to \mathcal{G}$, assigning to each arrow its inverse;
    \item Composition map 
    $\mathbf{m}: \mathcal{G}_2 \to \mathcal{G}$, where
    \begin{equation*}
        \mathcal{G}_2 := \left\{ (g, h) \in \mathcal{G} \times \mathcal{G} \;\middle|\; \mathbf{s}(g) = \mathbf{t}(h) \right\},
    \end{equation*}
    denotes the set of \textit{composable arrow pairs}.
\end{enumerate}
A groupoid $\mathcal{G}\rightrightarrows M$ is a \textbf{Lie groupoid} if $\mathcal{G}$ and $M$ are smooth manifolds, all structure maps are smooth, and $\mathbf{s}, \mathbf{t}$ are submersions.

We may consider cohomology of Lie groupoids. Let $\mathcal{G}_p$ denote the space of composable sequences of $p$-arrows. Then the face maps arising from the nerve $\mathcal{G}_{\bullet}$ form a simplicial manifold, which induces cohomology theories. In this paper, we will mainly work with two types of differentials: the groupoid differential
\[
\partial: \Omega^q(\mathcal{G}_{p})\rightarrow \Omega^q(\mathcal{G}_{p+1}),
\]
whose corresponding cohomology group ${\rm H}^{\bullet}(\mathcal{G},\Omega^q)$ is called the groupoid cohomology group with coefficients in $\Omega^q$; and further, the de Rham differential, 
\[
\delta:=\partial +(-1)^pd: \Omega^q(\mathcal{G}_p)\rightarrow \Omega^q(\mathcal{G}_{p+1})\oplus\Omega^{q+1}(\mathcal{G}_p),
\]
whose corresponding cohomology group is ${\rm H_{dR}^{\bullet}}(\mathcal{G})$. We remark that Lie groupoid cohomology with coefficients in other sheaves can also be defined, and an acyclic resolution is needed.

One important property of the de Rham cohomology is that it is invariant under Morita equivalence of Lie groupoids. Recall that a morphism of Lie groupoids $\phi:\mathcal{G}\rightarrow \mathcal{H}$ is called a Morita morphism if $\phi_0:\mathcal{G}_0\rightarrow \mathcal{H}_0$ is a surjective submersion, and the pullback Lie groupoid $\phi_0^!\mathcal{H}$ is isomorphic to $\mathcal{G}$. Two Lie groupoids $\mathcal{G}$ and $\mathcal{H}$ are called Morita equivalent if there exists a Lie groupoid $\mathcal{K}$ and Morita morphisms $\mathcal{K}$ to $\mathcal{G}$ and $\mathcal{K}$ to $\mathcal{H}$.

Following \cite{stacks_B-X}, we define a differentiable stack to be an equivalence class of Lie groupoids, up to Morita equivalence. We often use $[\mathcal{G}]$ for the differentiable stack represented by the Lie groupoid $\mathcal{G}$. Further, we take the corresponding de Rham cohomology as the differentiable stack cohomology.

An important class of Lie groupoids consists of the action groupoids arising from Lie group actions. Consider the conjugation action of a Lie group $G$ on itself, with corresponding action groupoid $G\times G\rightrightarrows G$. We denote the corresponding stack by $[G/G]$. The corresponding de Rham cohomology ${\rm H_{dR}^{\bullet}}([G/G])$ is identified with the equivariant cohomology ${\rm H_{G}^{\bullet}}(G)$.

We also need the notion of symplectic groupoids. Given $\omega\in \Omega^2(\mathcal{G})$, which is a symplectic form on a Lie groupoid $\mathcal{G}$, we say that $(\mathcal{G},\omega)$ is a symplectic groupoid if $\omega$ is multiplicative, i.e.
\[
\partial\omega:={\rm pr}_1^*\omega+{\rm pr}_2^*\omega-\mathbf{m}^*\omega = 0.
\]
Note that then $[\omega]$ also defines a $3$-class in ${\rm H_{dR}^3}([\mathcal{G}])$.

As an example, consider the coadjoint action groupoid $G\times \mathfrak{g}^*\rightrightarrows\mathfrak{g}^*$. We may identify $G\times \mathfrak{g}^*$ with the cotangent bundle $T^*G$ via left trivializations. Then $-\omega_{\rm can}$ makes $G\times \mathfrak{g}^*$ into a cotangent groupoid. Here we take the convention $\omega_{\rm can}=-d\theta_L$, where $\theta_L\in \Omega^1(T^*G)$ denotes the Liouville $1$-form.

A notable feature of  symplectic groupoids is that there is a unique Poisson structure on the base manifold such that the groupoid target map $\mathbf{t}$ is a Poisson map. For example, $(G\times \mathfrak{g}^*,-\omega_{\rm can})$ determines the Lie-Poisson structure on $\mathfrak{g}^*$.

Finally, we recall a family of Hamiltonian spaces. If $G$ acts on a manifold $M$, then the action of $G$ on $T^*M$, induced by the cotangent lift, is Hamiltonian with respect to $-\omega_{\rm can}$. Here, the moment map $\mu$ is induced by pairing with the Liouville $1$-form, and we adopt the convention $i_{\rho(v)}(-\omega_{\rm can}) = -d\mu^v$ for the moment map condition. Here $v\in \mathfrak{g}$ and $\rho: \mathfrak{g} \rightarrow \mathfrak{X}(M)$ is the infinitesimal action, given by
\[
\rho(v)|_m = \frac{d}{dt}\bigg|_{t=0}\exp(-tv)\cdot m,
\]
for all $v\in\mathfrak{g}$ and $m\in M$. Notice that if we consider $M=\mathfrak{g}^*$ with the coadjoint $G$-action, and identify $T^*G$ with $G\times \mathfrak{g}^*$ via left trivializations, then $(G\times \mathfrak{g}^*,-\omega_{\rm can})$ is a Hamiltonian $G$-space whose moment map coincides with the groupoid target map $\mathbf{t}: G\times \mathfrak{g}^*\rightarrow \mathfrak{g}^*$.

\section{Symplectic groupoids of  affine Poisson structures}\label{Sec:lie-algebra-case}

In this section, we recover the symplectic groupoid of the affine Poisson manifold $(\mathfrak g^*,\pi_\lambda)$ by two methods. The first gives an explicit realization as the affine action groupoid equipped with the canonical $2$-form with magnetic term. The second, while requiring the stronger assumption of a group central extension, is essential for our purposes since it naturally produces the $\mathbb S^1$-central extension of Lie groupoids used later in the gerbe construction. Symplectic groupoids of affine Poisson structures were studied earlier by Lu \cite{MR2685337} by a different method.

Let $G$ be a connected Lie group with Lie algebra  $\g$.
Let $\lambda \in \wedge^2 \mathfrak{g}^*$ be a Lie algebra 2-cocycle, i.e., $\lambda(u,[v,w])+ \lambda(v,[w,u]) + \lambda(w,[u,v]) = 0$, for any $u,v,w,\in \mathfrak{g}$. We denote by $\pi_{\lambda}$ the \textbf{affine Poisson structure} on $\g^*$ associated with $\lambda$, whose Poisson bracket is given by
\begin{equation}\label{Eqt:affine-lie-algebra}
	\{l_{v_1},l_{v_2}\} = l_{[v_1,v_2]}+\lambda(v_1,v_2),\quad \forall v_1,v_2\in \mathfrak{g}. 
\end{equation}
Here $l_{v_i}$ $(i=1,2)$ denotes the linear function on $\g^*$ defined by $l_{v_i}(\xi):=\langle \xi,v_i\rangle$ for all $\xi\in\g^*.$ 
We further define the map  $\lambda^{\flat}:\mathfrak{g} \rightarrow \mathfrak{g}^*$ by the relation $\langle \lambda^{\flat}(u), v\rangle=\lambda(u,v)$, for all $u,v\in \mathfrak{g}.$ It is obvious that $\lambda^{\flat}$ satisfies 
$$\ad^*_v \lambda^{\flat}(w)-\ad^*_w \lambda^{\flat}(v)=\lambda^{\flat}([v,w]),\quad \forall v,w\in\g.$$ 
This implies that  $\lambda^{\flat}$ is a 1-cocycle in the Chevalley-Eilenberg complex for the coadjoint representation $\g\to {\rm End}(\g^*)$, so its cohomology class $[\lambda^{\flat}]\in {\rm H}^1_{\ad^*}(\g,\g^*)$ is well-defined.

\subsection{First approach: action groupoid for the affine action}\label{Sec:first-construction}
\noindent

Let  $\Psi\colon {\rm H}_{\rm \Ad^*}^1(G,\g^*)\to {\rm H}_{\rm \ad^*}^1(\g,\g^*)$ denote the  Van Est morphism. On the cochain level, $\Psi$ is defined  by $$\Psi(f)(u)=\frac{d}{dt}|_{t=0} f(\exp(tu)),\quad f\in C^1(G,\g^*), u\in\g.$$ Suppose that the map $\lambda^{\flat}$ integrates to a $\g^*$-valued Lie group 1-cocycle $\chi: G\rightarrow \g^*$, that is, $\chi$ satisfies
\begin{equation}\label{Eqt:chi}
	\chi(gh)=\Ad^*_{g^{-1}} \chi(h)+\chi(g),\quad \forall g,h\in G,
\end{equation}
and  $\Psi([\chi])=[\lambda^{\flat}]$.
A key result in Van Est theory states that if 
$G$ is connected and simply connected, then the Van Est morphism  $\Psi$ is an isomorphism \cite{MR2016690}. In this case,  $\lambda^{\flat}$ always admits an integration.

Following the conventions outlined in \cite{marsdenbook}, the \textbf{affine action} of group $G$ on the dual space $\g^*$ associated with the 1-cocycle $\chi$  is defined as follows:
\begin{equation}\label{Eqt:affine-action}
	g\cdot \xi = \Ad^*_{g^{-1}}\xi -  \chi(g),
\end{equation}
for all $ g\in G$ and $\xi \in \mathfrak{g}^*$. 
Let $\Gamma:= G \times \g^*$ denote the transformation groupoid corresponding to the action given by Equation \eqref{Eqt:affine-action}.

Let $p_G: T^*G \rightarrow G$ be the bundle projection, and  $\lambda^L \in \Omega^2(G)$  the left-invariant 2-form associated to $\lambda$. Then the \textbf{canonical form with magnetic term} \cite{marsdenbook} is the 2-form $\omega_M\in \Omega^2(T^*G)$ given by
\begin{equation*}
	\omega_M = \omega_{\rm can}+p_G^*\lambda^L.
\end{equation*}

Let $\varphi: G\times \g^* \cong T^*G$ be the trivialization induced by left translations. Then  the 2-form $\omega_{\Gamma}: = -\varphi^*\omega_M \in \Omega^2(\Gamma)$  can be described by the following formula:
\begin{equation} \label{eq:O_g-def}
	\omega_{\Gamma}((v_1,\xi_1),(v_2,\xi_2))_{(g,\eta)} = \langle \xi_1, v_2 \rangle -  \langle \xi_2, v_1 \rangle - \langle \eta,[v_1,v_2]_{\mathfrak{g}}\rangle - \lambda(v_1,v_2),
\end{equation}
 for all $(v_i,\xi_i) \in \mathfrak{g}\times \mathfrak{g}^* \cong T_{(g,\eta)}\Gamma. $ The main result in this section is the following.

 \begin{thm}\label{thm: symplectic_groupoid}
 Let $G$ be a simply connected Lie group with Lie algebra $\mathfrak{g}$. Then $(\Gamma, \omega_\Gamma)$ is a symplectic groupoid integrating the affine Poisson structure $(\mathfrak{g}^*,\pi_{\lambda})$.    
 \end{thm}

We need the following lemma.

\begin{lemma}\label{lem:integrate-lemma}
Suppose that the Lie group 1-cocycle $\chi\colon G\to \g^*$ integrates $\lambda^{\flat}\colon \g\to \g^*$. Then the following identity holds:
$$\langle \chi(h), [v_1,v_2]\rangle - \lambda(v_1,v_2)
			+\lambda(\Ad_{h^{-1}}v_1,\Ad_{h^{-1}}v_2)=0,$$
            for all $v_1,v_2\in\g$ and $h\in G.$
\end{lemma}
\begin{proof}
A straightforward computation yields the following:
		\begin{align*}			&\lambda(\Ad_{h^{-1}}v_1,\Ad_{h^{-1}}v_2)\\
			&\quad= \frac{d}{dt}|_{t=0}\langle \chi(h^{-1}\exp(tv_1)h), \Ad_{h^{-1}}v_2\rangle \\
            &\quad = \frac{d}{dt}|_{t=0} \langle \Ad_h^*\chi(\exp(tv_1)h) + \chi(h^{-1})\rangle, \Ad_{h^{-1}} v_2\rangle \\
            & \quad= \frac{d}{dt}|_{t=0} \langle \Ad_h^*\Ad_{\exp(-tv_1)}^*\chi(h) + \Ad_h^* \chi(\exp(tv_1)) + \chi(h^{-1}),\Ad_{h^{-1}}v_2\rangle 
            \\
			&\quad=-\langle \chi(h), [v_1,v_2]\rangle + \lambda(v_1,v_2).
		\end{align*}

\end{proof}

\begin{propn}\label{prop: multiplicative}
	The symplectic form $\omega_{\Gamma}$ is multiplicative on $\Gamma$, that is, ${\rm pr}_1^*\omega_{\Gamma}+{\rm pr}_2^*\omega_{\Gamma}-\mathbf{m}_{\Gamma}^*\omega_{\Gamma} = 0.$  
\end{propn}

\begin{proof}
We first establish an identification between 
	$\Gamma_2$ and $G\times G \times \mathfrak{g}^*$ via the mapping $$\big((g,\beta),(h,\eta)\big) \mapsto (g,h,\eta),\quad \text{for all}~ g,h\in G,\quad \beta,\eta\in\g^*.$$
    Under this identification, the projection maps ${\rm pr}_1, {\rm pr}_2\colon \Gamma_2\to \Gamma$  and the groupoid multiplication $\textbf{m}_\Gamma\colon \Gamma_2\to \Gamma$  take explicit forms:
 $${\rm pr}_1(g,h,\eta) = (g,h\cdot \eta), \quad {\rm pr}_2(g,h,\eta) = (h, \eta),\quad \textbf{m}_{\Gamma}(g,h,\eta) = (gh,\eta),$$
 for all $g,h\in G$ and $\eta\in\g^*.$

For any $(v,w,\xi)\in \mathfrak{g}\times \mathfrak{g}\times\mathfrak{g}^* \cong T_{(g,h,\eta)}(G\times G\times \mathfrak{g}^*)$,
we compute the pushforward ${{\rm pr}_1}_*(v,w,\xi)$ 
 as follows:
    \begin{align*}
     {{\rm pr}_1}_*(v,w,\xi)&= \frac{d}{dt}|_{t=0}{\rm pr}_1\big(g\exp(tv),h\exp(tw),\eta+t\xi\big) \\
    &=\frac{d}{dt}|_{t=0}\big(g\exp(tv),h\exp(tw)\cdot(\eta+t\xi)\big)\\
    &=\big(v,  \frac{d}{dt}|_{t=0}h\exp(tw)\cdot \eta + {\rm Ad}_{h^{-1}}^*\xi\big)\\
    &=\big(v,  \frac{d}{dt}|_{t=0} {\rm Ad}_{h^{-1}}^* ({\rm Ad}_{\exp(-tw)}^*  \eta - \chi(\exp(tw)))  + \rm Ad_{h^{-1}}^*\xi\big)\\
    &= \big(v, \Ad^*_{h^{-1}}(\xi-\ad^*_{w}\eta - \lambda^{\flat}(w))\big).    
    \end{align*}  
Similarly, the pushforward of the projection map ${\rm pr}_1$ and multiplication map $\mathbf{m}_{\Gamma}$ are:
	\begin{equation*}
		\begin{aligned}
			{{\rm pr}_2}_*(v,w,\xi)&=(w,\xi), \\
			{\mathbf{m}_{\Gamma}}_*(v,w,\xi)&= (w+\Ad_{h^{-1}}v, \xi).
		\end{aligned}
	\end{equation*}

Using the explicit formula for $\omega_\Gamma$
 in Equation \eqref{eq:O_g-def}, we compute the pullbacks of $\omega_\Gamma$  along ${\rm pr}_1, {\rm pr}_2$ and $\mathbf{m}_{\Gamma}$  for arbitrary tangent vectors $(v_1,w_1,\xi_1),(v_2,w_2,\xi_2)\in T_{(g,h,\eta)}(G\times G\times \mathfrak{g}^*)$. First, the pullback along ${\rm pr}_1$:
\begin{align*}
			&({\rm pr}_1^*\omega_{\Gamma})((v_1,w_1,\xi_1),(v_2,w_2,\xi_2))\\
			=& \langle \xi_1, \Ad_{h^{-1}}v_2 \rangle -\langle \eta, [w_1,\Ad_{h^{-1}}v_2] \rangle - \lambda (w_1,\Ad_{h^{-1}}v_2)\\
			&-\langle \xi_2, \Ad_{h^{-1}}v_1 \rangle +\langle \eta, [w_2,\Ad_{h^{-1}}v_1] \rangle + \lambda (w_2,\Ad_{h^{-1}}v_1) \\
			& - \langle \Ad_{h^{-1}}[v_1,v_2] , \eta \rangle + \langle \chi(h), [v_1,v_2]\rangle - \lambda(v_1,v_2).
		\end{align*}
Next, the pullback along ${\rm pr}_2$:
	\begin{equation*}
		\begin{aligned}
			&({\rm pr}_2^*\omega_{\Gamma})((v_1,w_1,\xi_1),(v_2,w_2,\xi_2))\\
			=& \langle w_2, \xi_1 \rangle - \langle w_1, \xi_2 \rangle -\langle [w_1,w_2], \eta \rangle - \lambda(w_1,w_2).
		\end{aligned}
	\end{equation*}
	Finally, the pullback along 
$\textbf{m}_\Gamma$:
	\begin{equation*}
		\begin{aligned}
			&(\textbf{m}_{\Gamma}^*\omega_{\Gamma})((v_1,w_1,\xi_1),(v_2,w_2,\xi_2))\\
			=& \langle w_2, \xi_1 \rangle - \langle w_1, \xi_2 \rangle +\langle \Ad_{h^{-1}}v_2, \xi_1 \rangle - \langle \Ad_{h^{-1}}v_1, \xi_2 \rangle\\
			&-  \langle [w_1,w_2], \eta \rangle - \langle \Ad_{h^{-1}}[v_1,v_2],\eta \rangle - \langle \eta, [w_1,\Ad_{h^{-1}}v_2]\rangle + \langle \eta, [w_2,\Ad_{h^{-1}}v_1]\rangle \\
			&- \lambda(w_1,w_2) - \lambda(w_1,\Ad_{h^{-1}}v_2) - \lambda(\Ad_{h^{-1}}v_1, w_2) - \lambda(\Ad_{h^{-1}}v_1,\Ad_{h^{-1}}v_2).
		\end{aligned}
	\end{equation*}

By substituting these pullbacks into the combination ${\rm pr}_1^*\omega_{\Gamma}+{\rm pr}_2^*\omega_{\Gamma}-\textbf{m}_{\Gamma}^*\omega_{\Gamma}$ 
 and simplifying using Lemma \ref{lem:integrate-lemma}, we find that all cross terms cancel, leaving:
	\begin{equation*}
		\begin{aligned}
			&({\rm pr}_1^*\omega_{\Gamma}+{\rm pr}_2^*\omega_{\Gamma}-\textbf{m}_{\Gamma}^*\omega_{\Gamma})((v_1,w_1,\xi_1),(v_2,w_2,\xi_2))\\
			&\quad= \langle \chi(h), [v_1,v_2]\rangle - \lambda(v_1,v_2)
			+\lambda(\Ad_{h^{-1}}v_1,\Ad_{h^{-1}}v_2)\\ &\quad =0.
		\end{aligned}
	\end{equation*}
So we conclude that $\omega_\Gamma$ is multiplicative.	
	
\end{proof}

To establish that $(\Gamma, \omega_{\Gamma})$ is the symplectic groupoid of $(\g^*,\pi_{\lambda})$, a key requirement is that $\mathbf{t}_\Gamma: (\Gamma, \omega_{\Gamma})\rightarrow (\g^*,\pi_{\lambda})$ constitutes a symplectic realization. Indeed, this is equivalent to verifying that $\mathbf{s}_{\Gamma}: (\Gamma, \omega_{\Gamma})\rightarrow (\g^*,-\pi_{\lambda})$ is a symplectic realization. A detailed proof of this result can be found in \cite{marsdenbook}; here, we present a concise sketch for the sake of completeness.

\begin{propn}\label{prop :symp_realization} 
	The groupoid source map $\mathbf{s}_{\Gamma}:(\Gamma,\omega_{\Gamma})\rightarrow (\g^*,-\pi_{\lambda})$ is a symplectic realization. 
\end{propn}

\begin{proof}
By the definition of $\omega_{\Gamma}$ given in Equation \eqref{eq:O_g-def}, we have 
	$$\omega_\Gamma^{-1}(\mathbf{s}_\Gamma^* dl_{v})_{(g,\xi)}=(-v, -\ad^*_v\xi-\lambda^\flat(v)),$$
	for any $(g,\xi)\in\Gamma$ and $v\in\g.$ Now, for any $v_1,v_2\in\g$, one can derive
	\begin{align*}
		&\pi_{\Gamma}(\mathbf{s}_\Gamma^* dl_{v_1},\mathbf{s}_\Gamma^* dl_{v_2})_{(g,\xi)}\\
		=&-\omega_\Gamma\big((-v_1, -\ad^*_{v_1}\xi-\lambda^\flat(v_1)),(-v_2, -\ad^*_{v_2}\xi-\lambda^\flat(v_2))\big)\\
		=&-\langle \xi,[v_1,v_2]\rangle-\lambda(v_1,v_2)\\
		=&-\pi_{\lambda}(dl_{v_1},dl_{v_2})_\xi.
	\end{align*}
This directly implies that $(\mathbf{s}_\Gamma)_*\pi_{\Gamma}=-\pi_{\lambda},$ which is the desired conclusion.
	
\end{proof}

Combining Propositions \ref{prop: multiplicative} and \ref{prop :symp_realization}, we have established Theorem \ref{thm: symplectic_groupoid}.

\subsection{Second approach: reduction of cotangent bundle}\label{Sec:Second-construction}
\noindent 

Now we present an alternative approach for the construction of symplectic groupoid, adapting the method developed in  \cite{B-X-Z}. The idea is to recognize the affine Poisson manifold $(\g^*,\pi_{\lambda})$  as a Poisson submanifold of some dual Lie algebra $\widetilde{\g}^*$ equipped with the standard Lie-Poisson structure $\widetilde{\pi}_{\rm Lie}$. Then the symplectic groupoid $\Gamma$ of $\pi_{\lambda}$  can be constructed via symplectic reduction of the symplectic groupoid associated to $(\widetilde{\g}^*, \widetilde{\pi}_{\rm Lie})$. Notably, this method start with a stronger assumption, namely the existence of a group extension, but the construction yields naturally a central extension of Lie groupoids, which play an important role in later sections.

Any Lie algebra 2-cocycle $\lambda \in \wedge^2\g^*$ determines a $\mathbb{R}$-central extension of the Lie algebra $\g$, as follows. Endow $\widetilde{\mathfrak{g}}=\mathfrak{g}\oplus\mathbb{R}$ with the Lie brackets
\begin{equation*}
	[(v_1,r_1),(v_2,r_2)]_{\widetilde{\mathfrak{g}}}=([v_1,v_2]_{\mathfrak{g}},\lambda(v_1,v_2)),    
\end{equation*}
for any $(v_i,r_i)\in \widetilde{\mathfrak{g}}=\mathfrak{g}\oplus\mathbb{R}.$ Then it fits into the short exact sequence 
\begin{equation}\label{SES_Lie_Alg}
	0\rightarrow \mathbb{R}\rightarrow \widetilde{\g} \rightarrow \g\rightarrow 0     
\end{equation}
where the second arrow denote the inclusion into $\mathbb{R}$-component of $\widetilde{\g}$, and the third arrow is the projection onto $\g$. We can  embed $\mathfrak{g}^*$ in $\widetilde{\mathfrak{g}}^*$ via the map $j(\xi) := (\xi, 1).$ 
Then $(\mathfrak{g}^*, \pi_{\lambda})$ is a Poisson submanifold of $(\widetilde{\mathfrak{g}}^*, \widetilde{\pi}_{\rm Lie})$.

Regarding the central extension of  Lie algebras \eqref{SES_Lie_Alg}, let us recall the following result concerning its integration.

\begin{lemma}[\cite{Neeb}, Theorem I.2]\label{lem:Neeb}
	Assume that $G$ is a simply connected Lie group with Lie algebra $\g$, then there exists a Lie group $\widetilde{G}$ with Lie algebra $\widetilde{\mathfrak{g}}$, and a central extension 
	\begin{equation}\label{SESLieGrp}
		1 \xrightarrow[]{} \mathbb{S}^1 \xrightarrow[]{i}\widetilde{G} \xrightarrow[]{p} G \xrightarrow[]{}  1,  
	\end{equation}
	integrating the central extension \eqref{SES_Lie_Alg}.
\end{lemma}

From now on, we assume the existence of central extension \eqref{SESLieGrp}. Next, we investigate the relation between the symplectic groupoid of $(\mathfrak{g}^*, \pi_{\lambda})$ and $(\widetilde{\mathfrak{g}}^*, \widetilde{\pi}_{\rm Lie})$. By Section \ref{Sec:preliminary}, the cotangent groupoid $(T^*\widetilde{G},-\widetilde{\omega}_{\rm can})$ is symplectic groupoid of $(\widetilde{\g}^*,\widetilde{\pi}_{\rm Lie})$. Let $\widetilde{\varphi}: \widetilde{G}\times \widetilde{\g}^*\rightarrow T^*\widetilde{G}$ be the diffeomorphism induced by left translations, then the groupoid $T^*\widetilde{G}$ can be viewed as the transformation groupoid $\widetilde{G}\times \widetilde{\g}^*$ with respect to the coadjoint action of $\widetilde{G}$ on $\widetilde{\g}^*$.  For simplicity, we denote the groupoid $\widetilde{G}\times \widetilde{\g}^*$ as $\widetilde{\Gamma}$.

In view of the central extension \eqref{SESLieGrp}, $\mathbb{S}^1$ is a Lie subgroup of $\widetilde{G}$, which induces a natural $\mathbb{S}^1$-action on $\widetilde{G}$ by left multiplication. This action lifts to an 
$\mathbb{S}^1$-action on the cotangent bundle $T^*\widetilde{G}$, thereby making $\widetilde{\Gamma}\cong T^*\widetilde{G}$ a Hamiltonian $\mathbb{S}^1$-space. We now explicitly describe this action and its associated moment map on $\widetilde{\Gamma}.$
The action map $\widetilde{F}:\mathbb{S}^1 \times \widetilde{\Gamma}\rightarrow \widetilde{\Gamma}$ is given by $\widetilde{F}(s,(\widetilde{g},\widetilde{\eta})) = (s\widetilde{g},\widetilde{\eta}),$ for all $s\in\mathbb{S}^1$ and $(\widetilde{g},\widetilde{\eta})\in\widetilde{\Gamma}.$ Let ${\rm pr}_{\mathbb{R}}:\widetilde{\g}^*=\g^*\oplus \mathbb{R}\rightarrow \mathbb{R}$ be the projection. The moment map $\widetilde{\mu}:\widetilde{\Gamma} \rightarrow \mathbb{R}$ is given by $\widetilde{\mu}={\rm pr}_{\mathbb{R}}\circ \widetilde{\mathbf{t}}$. Here $\widetilde{\mathbf{t}}$ denote the target map of the groupoid $\widetilde{G}\times \widetilde{\g}^*.$ 
According to Section \ref{Sec:preliminary},
$(\widetilde{\Gamma}, \widetilde{F}, \widetilde{\mu}, -\widetilde{\varphi}^*\widetilde{\omega}_{\rm can})$ is a Hamiltonian $\mathbb{S}^1$-space.

Observe that $1$ is a regular value of $\widetilde{\mu}$, and $\widetilde{\mu}^{-1}(1) = \widetilde{G}\times (\g^*\times\{1\})$. Since $\mathbb{S}^1$ only act on the $\widetilde{G}$-component of $\widetilde{\Gamma}$,  it follows that the quotient space $\widetilde{\mu}^{-1}(1)/\mathbb{S}^1$ can be expressed as $(\widetilde{G}/\mathbb{S}^1)\times (\g^*\times\{1\})$, which is diffeomorphic to $G\times \g^*=\Gamma$ as smooth manifolds. Indeed, the groupoid structure on $\widetilde{\mu}^{-1}(1)/\mathbb{S}^1$ also  coincides with the transformation groupoid ${\Gamma}$. 
\begin{propn}\label{propn: Ham_space1}
	The transformation groupoid $\widetilde{\Gamma}$ has a natural structure of Hamiltonian $\mathbb{S}^1$-space. At the regular value $1$ of the moment map, it reduces to the groupoid $\Gamma$.
\end{propn}
\begin{proof}
Let $\widetilde{\rm Ad}$ and  $\rm Ad$ denote the adjoint representation of $\widetilde{G}$ and $G$, respectively. Given the $\mathbb{S}^1$-central extension \eqref{SESLieGrp}, the corresponding group 1-cocycle integrating $\lambda^{\flat}$  is represented by a map $\chi: G \rightarrow \g^*$, defined by the equation  $$\langle \chi(p(\widetilde{h})), v\rangle := -{\rm pr}_{\mathbb{R}}\circ \widetilde{\Ad}_{\widetilde{h}^{-1}}(v,0)
	,$$ for any $\widetilde{h}\in \widetilde{G}$ and $ v\in \g$. 
From this definition, we derive a formula that elucidates the relation between the adjoint representations of $\widetilde{\rm Ad}$ and  $\rm Ad$. Specifically, for any $\widetilde{h}\in \widetilde{G}$ and $(v,r) \in \widetilde{\mathfrak{g}}=\mathfrak{g}\oplus\mathbb{R}$, we have   
	\begin{equation*}\label{eq:adj-action}
		\widetilde{\Ad}_{\widetilde{h}}(v,r) = (\Ad_{p(\widetilde{h})}v, r-\langle \chi(p(\widetilde{h})^{-1}), v\rangle).   
	\end{equation*}

Next, for any $\widetilde{h}\in\widetilde{G}$ and $(\xi,t) \in \widetilde{\mathfrak{g}}^*= \mathfrak{g}^*\oplus \mathbb{R},$ the coadjoint action is given by
	\begin{equation*}\label{eq:coadj-action}
		\widetilde{\Ad}^*_{\widetilde{h}^{-1}}(\xi,t) = (\Ad^*_{p(\widetilde{h})^{-1}}\xi - t\chi(p(\widetilde{h})), t) \in \mathfrak{g}^*\oplus \mathbb{R}.
	\end{equation*}
By setting $t=1$, we focus on the restricted action of $\widetilde{G}$ on $\g^*\times \{1\}$. This restricted action provides the groupoid structure on $\widetilde{G}\times \g^*$, which is isomorphic to $\widetilde{\mu}^{-1}(1)$. Upon taking the quotient by $\mathbb{S}^1$, this action transforms into the affine action as specified in Equation \eqref{Eqt:affine-action}. Consequently, we conclude that $\widetilde{\mu}^{-1}(1)/\mathbb{S}^1$ is isomorphic to $\Gamma$ as Lie groupoids.
\end{proof}

\begin{remark}
Based on the above discussions, given an $\mathbb{S}^1$-central extension of Lie groups $\widetilde{G}$ over $G$, one can get a $\mathbb{S}^1$-central extension of Lie groupoids $\widetilde{G}\times \g^*$ over $G\times \g^*$. See Section \ref{Sec:5.1}.  
\end{remark}

Next, we aim to demonstrate that the reduced symplectic form of $-\widetilde{\varphi}^*\widetilde{\omega}_{\rm can} \in \Omega^2(\widetilde{\Gamma})$ coincides with $\omega_{\Gamma}$ as  in Equation \eqref{eq:O_g-def}. To this end, we define the projection $\widetilde{p}: \widetilde{\mu}^{-1}(1) \rightarrow \Gamma$ by $\widetilde{p}(\widetilde{h},(\eta,1)) = (p(\widetilde{h}),\eta)$, and  introduce the natural inclusion $\widetilde{j}: \widetilde{\mu}^{-1}(1)\rightarrow \widetilde{\Gamma}$. The following result for symplectic reduction of cotangent bundle at regular value $1$ is classical, and its proof can be found in \cite{marsdenbook}. We sketch a proof here for completion.

\begin{propn}\label{prop:rel_OG&OC}
On $\widetilde{\mu}^{-1}(1),$	the formula  $-\widetilde{j}^*\widetilde{\varphi}^*\widetilde{\omega}_{\rm can} = \widetilde{p}^*\omega_{\Gamma}$ holds, which implies that the symplectic form $-\widetilde{\varphi}^*\widetilde{\omega}_{\rm can} \in \Omega^2(\widetilde{\Gamma})$ reduces to $\omega_{\Gamma}$ on $\Gamma \cong\widetilde{\mu}^{-1}(1)/\mathbb{S}^1$. Thus, the reduced symplectic groupoid is ismoorphic to the symplectic groupoid discussed in Section \ref{Sec:first-construction}.
\end{propn}

\begin{proof}
	Since $T_{(\widetilde{g},(\eta,1))}\widetilde{\mu}^{-1}(1)\cong \widetilde{\g}\times \g^*$, for any $(\widetilde{v}_i,\xi_i)\in \widetilde{\g}\times \g^*$,  we compute by Equation \eqref{eq:O_g-def} that
	\begin{equation*}
		\begin{aligned}
			&\widetilde{p}^*\omega_{\Gamma}\big((\widetilde{v}_1,\xi_1),(\widetilde{v}_2,\xi_2)\big)_{(\widetilde{g},(\eta,1))}\\
			=&\omega_{\Gamma}((v_1,\xi_1),(v_2,\xi_2))_{(p(\widetilde{g}), \eta)} \\
			=& \langle \xi_1, v_2 \rangle -  \langle \xi_2, v_1 \rangle - \langle \eta,[v_1,v_2]_{\mathfrak{g}}\rangle - \lambda(v_1,v_2).
		\end{aligned}
	\end{equation*}
	Here $\widetilde{v}_i$ is decomposed as $v_i+r_i$.
	On the other hand,
	\begin{equation*}
		\begin{aligned}
			& \widetilde{j}^*\widetilde{\varphi}^*\widetilde{\omega}_{\rm can}\big((\widetilde{v}_1,\xi_1),(\widetilde{v}_2,\xi_2)\big)_{(\widetilde{g},(\eta,1))}\\
			=& \widetilde{\varphi}^*\widetilde{\omega}_{\rm can}\big((\widetilde{v}_1,(\xi_1,0)),(\widetilde{v}_2,(\xi_2,0))\big)_{(\widetilde{g},(\eta,1))}\\
			=& -\langle \xi_1, v_2\rangle + \langle \xi_2, v_1\rangle+ \langle [\widetilde{v}_1,\widetilde{v}_2]_{\widetilde{\mathfrak{g}}}, (\eta,1)\rangle\\
			=& -\langle \xi_1, v_2 \rangle +  \langle \xi_2, v_1 \rangle + \langle \eta,[v_1,v_2]_{\mathfrak{g}}\rangle + \lambda(v_1,v_2).
		\end{aligned}
	\end{equation*}
	Hence, the claim is established.
\end{proof}

\begin{remark}
	In the work of Behrend-Xu-Zhang \cite{B-X-Z}, a slightly different approach was sketched for obtaining the symplectic form $\omega_{\Gamma}$. Namely, instead of constructing $\omega_{\Gamma}$ directly as Equation \eqref{eq:O_g-def}, one can show that the form $\widetilde{j}^*\widetilde{\omega}_{\rm can}$ on $R:=\widetilde{G}\times \g^*$ is basic with respect to the $\mathbb{S}^1$-bundle $R\rightarrow \Gamma.$ Then one can go through the proof of Proposition \ref{prop:rel_OG&OC} and conclude that the unique form one get on $\Gamma$ is exactly $\omega_{\Gamma}$.    
\end{remark}

\section{The Case of Loop Lie algebras}\label{Sec:loop-case}
In this section, we apply the construction presented in Section \ref{Sec:first-construction} to the case of loop Lie algebras. Furthermore, we delve into the relation between the reduced symplectic form  and the Alekseev-Malkin-Meinrenken (AMM) cocycle on $G\times G\rightrightarrows G$.

\subsection{Loop groups and AMM cocycle}

Throughout the whole section, we assume that $G$ is a compact, connected Lie group, and let $\mathfrak{g}$ be its Lie algebra. Following the notation introduced in \cite{A-M-M}, we denote 
$LG:={\rm Map}(\mathbb{S}^1,G)$ to be the loop group of $G$, where the multiplication on $LG$ is defined to be the pointwise multiplication. Its Lie algebra is given by $L\mathfrak{g} := \Omega^0(\mathbb{S}^1,\mathfrak{g})$.

Let $(\cdot,\cdot)$ be a symmetric invariant bilinear form on $\mathfrak{g}$, where invariance means $([u,v],w) = (u,[v,w]),$ for all $u,v,w\in\g.$ 
We define $L\g^*$ as the space of $1$-forms $\Omega^1(\mathbb{S}^1,\g)= \Omega^1(\mathbb{S}^1)\otimes \g.$ Next, we define a map $\lambda:\wedge^2L\g \rightarrow \mathbb{R}$ by \begin{equation}\label{eq:loopcocycle}
	\lambda(v_1,v_2) = \int_{\mathbb{S}^1}(v_1, d{v_2}).
\end{equation}  Consequently, $\lambda$ induces a map  $\lambda^{\flat}: L\g\rightarrow L\g^*$, given by $\lambda^{\flat}(v) = -d{v}$.

 There is a Van Est morphism
\[
\Psi\colon {\rm H}^1(LG,L\mathfrak g^*) \longrightarrow {\rm H}^1(L\mathfrak g,L\mathfrak g^*)
\]
defined at the cochain level by
\[
\Psi(f)(v)=\left.\frac{d}{dt}\right|_{t=0} f(\exp(tv)),
\]
for all $v\in L\mathfrak g$ and $f\in C^1(LG,L\mathfrak g^*)$. Although the Van Est map need not be an isomorphism in degree one for a general connected Lie group, the Lie algebra cocycle $\lambda^{\flat}$ integrates explicitly in this case. Indeed, it is the image under the Van Est map of the group $1$-cocycle $\chi\colon LG\to L\mathfrak g^*$ defined by $\chi(\gamma)=-(d\gamma)\gamma^{-1}$.

The loop group $LG$ further acts on $L\g^*$ via the gauge action, given by
\begin{equation} \label{eq: LH_on_(Lh)*}
	\gamma \cdot A = \Ad_{\gamma}A +  (d{\gamma})\gamma^{-1},
\end{equation}
for any $\gamma \in LG$ and $A \in L\mathfrak{g}^*$. Therefore, employing the same approach as discussed in Section \ref{Sec:first-construction}, we obtain a symplectic groupoid $(LG \times L\g^*, \omega_{LG \times L\g^*})$.

\begin{propn}\label{thm: LGxLg-symgrpd}
	Let $LG\times L\g^*$ be the transformation groupoid associated to the gauge action defined in Equation \eqref{eq: LH_on_(Lh)*}. For any tangent vector $(v_i,A_i) \in L\mathfrak{g}\times L\mathfrak{g}^* \cong T_{(\gamma,A)}(LG\times L\g^*), $ we describe the canonical 2-form with magnetic term
	\begin{equation} \label{eq:O_Lh-def}
		\begin{aligned}
			&\omega_{LG\times L\mathfrak{g}^*}((v_1,A_1),(v_2,A_2))_{(\gamma,A)} \\
			=& \int_{\mathbb{S}^1}(A_1, v_2 )-  ( A_2, v_1 ) - (A,[v_1,v_2]) - (v_1,d{v_2}).  
		\end{aligned}
	\end{equation} 
	Then $(LG\times L\g^*, \omega_{LG\times L\g^*})$ is a symplectic groupoid. 
\end{propn}

\begin{remark}
When $G$ is a compact, simply connected Lie group, then a group central extention of $LG$ exists, and the method from Section \ref{Sec:Second-construction} applies. In this case, we get an $\mathbb{S}^1$-central extension of Lie groupoids over $LG \times L\g^*$, and see Section \ref{sec: ThmA} for further discussions.
\end{remark}

Next we recall the AMM cocycle from \cite{A-M-M}. Consider the conjugation action of $G$ on itself, and let $D := G \times G$ be the corresponding arrow space of the transformation groupoid over $G$.
Denote by $\theta$ and $\bar{\theta}$  the left-invariant and right-invariant Maurer-Cartan form on $G$, respectively. Let $\Omega \in \Omega^3(G)$ be the invariant 3-form, given by $\Omega = \frac{1}{12}(\theta,[\theta, \theta])$.  The 2-form $\omega_D \in \Omega^2(D)$ is defined by
\begin{equation}\label{eq:O_D-def}
	{\omega_D}|_{(x,y)}:=\frac{1}{2}(\Ad_yp_1^*\theta, p_1^*\theta)+\frac{1}{2}(p_1^*\theta, p_2^*\theta+p_2^*\bar{\theta}).
\end{equation}
Here, $p_1,p_2:G\times G\rightarrow G$ are the natural projections for first and second components.

\begin{propn}[\cite{Xu}, Proposition 2.8]\label{prop:xu-prop}
	The class $[\omega_D + \Omega]$ is a well-defined class in $ {\rm H}^3_{\rm dR}(D)$, called the AMM class. 
\end{propn}

\subsection{Holonomy and inversion}
Before we establish the relation between $\omega_{LG \times L\mathfrak{g}^*}$ and $\omega_D$, we  recall two key maps: the holonomy map, ${\rm Hol}$, and the inversion map, ${\rm inv}.$ These maps are also used in \cite{A-M-M} and \cite{meinrenken1996symplecticproof}

\subsubsection{The holonomy map}
Recall that for a principal bundle with a connection, the holonomy of a loop on the base manifold is defined as the parallel transport along that loop. In the special case of the trivial $G$-bundle over $\mathbb{S}^1$, the space $\Omega^1(\mathbb{S}^1,\g)$ of connections  can be identified with $L\g^*$. By picking the identity loop ${\rm id}_{\mathbb{S}^1}$ and computing its holonomy, we obtain a map $\rm Hol: L\g^* \rightarrow G$, called \textbf{the holonomy map}.

Let $\Omega G\subset LG$ be the set of loops based at the identity element $e\in G$. 
\begin{lemma}\cite{A-M-M}
	The restricted action of $\Omega G$ on $L\mathfrak{g}^*$ from Equation \eqref{eq: LH_on_(Lh)*} makes $L\mathfrak{g}^*$ a principal $\Omega G$-bundle over $G$. The bundle projection is given by the holonomy map ${\rm Hol}: L\mathfrak{g}^* \rightarrow G$.
\end{lemma}

An equivalent, analytic description of the holonomy map is often useful. We identity $\mathbb{S}^1\cong \R/\Z$ and denote by $s\in\R$ the local coordinate.
Let ${\rm Hol}_s(A):[0,1]\rightarrow G$ be the solution to the ODE \cite{A-M-M}:
\begin{equation}\label{eq:ODE_Hol}
	{\rm Hol}_s(A)^{-1}\frac{d}{d s} {\rm Hol}_s(A) = A(s),
\end{equation}
with ${\rm Hol}_0(A) =e$. Then $\rm Hol: L\g^* \rightarrow G$ is given by $\rm Hol(A) = \rm Hol_1(A)$.

A crucial property of the holonomy map is its compatibility with the $LG$-action on $L\g^*$. Specifically, for any $\gamma\in LG$, and $ A \in L\mathfrak{g}^*$, we have:
\begin{equation*}
	{\rm Hol}(\gamma \cdot A) = \Ad_{\gamma(0)}{\rm Hol}(A),
\end{equation*}
where $\gamma \cdot A$ denotes the gauge action defined in Equation \eqref{eq: LH_on_(Lh)*}. Since $\Omega G$ is a normal subgroup of $LG$, we take an $\Omega G$-quotient to the $LG$ action on $L\g^*$.  The induced action on the quotient space $L\g^*/\Omega G\cong G$  coincides with the conjugation action of 
$G$ on itself. 
\subsubsection{The inversion map}
Next, we recall the notion of inversion maps, and summarize the compatibility properties between the holonomy map  and the inversion map. Let $I:[0,1]\rightarrow [0,1]$ denote the map $I(s) = 1-s.$
\begin{defn}
	The \textbf{inversion map} ${\rm inv}: L\g^* \rightarrow L\g^*$  is defined by $${\rm inv}(A)(s) = -A(I(s))=-A(1-s),\quad \forall A\in L\g^*,\quad s\in [0,1].$$
\end{defn}


\begin{lemma}\label{lem: invA}
For any $A\in L\g^*$, the holonomy of ${\rm inv}(A)$ satisfies
$${\rm Hol}_s({\rm inv}(A)) = {\rm Hol}_1(A)^{-1}{\rm Hol}_{1-s}(A).$$ In particular, ${\rm Hol}({\rm inv}(A))={\rm Hol}(A)^{-1}$.
\end{lemma}

\begin{proof}

It suffices to check that $H(s):={\rm Hol}_1(A)^{-1}{\rm Hol}_{1-s}(A)$ satisfies the ODE defined in Equation \eqref{eq:ODE_Hol}. Through direct computation, we obtain
\begin{align*}
& H(s)^{-1}\frac{d}{ds}H(s)\\
=&[{\rm Hol}_1(A)^{-1}{\rm Hol}_{1-s}(A)]^{-1}\frac{d}{dt}|_{t=s}[{\rm Hol}_1(A)^{-1}{\rm Hol}_{1-t}(A)]\\
=& -[{\rm Hol}_1(A)^{-1}{\rm Hol}_{1-s}(A)]^{-1}{\rm Hol}_{1}(A)^{-1}\frac{d}{dt}|_{t=1-s}{\rm Hol}_{t}(A)\\ 
=& -{\rm Hol}_{1-s}(A)^{-1}\frac{d}{dt}|_{t=1-s}{\rm Hol}_{t}(A)\\ 
=& -A(1-s) = {\rm inv}(A)(s).
\end{align*}

\end{proof}

\begin{lemma}\label{lem:inversion-action}
	For any $\gamma\in LG$ and $A \in L\mathfrak{g}^*$, in view of the action defined by Equation \eqref{eq: LH_on_(Lh)*}, we have
	\begin{equation*}
		\gamma \cdot {\rm inv}(A) = {\rm inv}((I^*\gamma)\cdot A).
	\end{equation*}
Here, $I^*\gamma\in LG$ denotes the pullback of 
$\gamma$ by $I$, defined by $I^*\gamma(s)=\gamma(I(s))=\gamma(1-s).$
\end{lemma}

\begin{proof}
By the chain rule, we have $\frac{d}{dt}I^*\gamma = -I^*\frac{d}{dt}\gamma$, and thus $(\frac{d}{dt}I^*\gamma)(I^*\gamma)^{-1} = -I^*(\dot{\gamma}\gamma^{-1})$. Since $I^2$ is the identity map, we get that $\dot{\gamma}\gamma^{-1} = -I^*(\frac{d}{dt}I^*\gamma)(I^*\gamma)^{-1}$. By definition of ${\rm inv}(A) =-I^*A $, we have
	\begin{equation*}
		\begin{aligned}
			&\gamma \cdot {\rm inv}(A) =  \gamma \cdot (-I^*A) 
			= \Ad_{\gamma}(-I^*A)+\dot{\gamma}{\gamma}^{-1}.
		\end{aligned}
	\end{equation*}
On the other hand,
\begin{align*}
 {\rm inv}((I^*\gamma)\cdot A)
 &= 			 {\rm inv}(\Ad_{I^*\gamma}A + (\frac{d}{dt}I^*\gamma)(I^*\gamma)^{-1})\\
 &= -   \Ad_{\gamma}I^*A - I^* (\frac{d}{dt}I^*\gamma)(I^*\gamma)^{-1}\\
 &= \Ad_{\gamma}(-I^*A)+\dot{\gamma}{\gamma}^{-1}.
\end{align*}
Thus the result follows.
\end{proof}

\subsection{The relation of $\omega_{LG\times L\mathfrak{g}^*}$ and $\omega_D$}

Now we establish the following identity.

\begin{thm}\label{thm:for-equiv}
The forms $\omega_{LG\times L\g*}$ and $\omega_D$ are related by
\begin{equation}\label{eq:form-rel}
\omega_{LG\times L\g*}=f^*\omega_D-\Phi^*(\varpi,\varpi).
\end{equation}
Here, the projection $			f:LG \times L\g^* \rightarrow G\times G$ is given by
\[
f(\gamma,A)= (\gamma(0), \rm{Hol}(A)).
\]
The map $\Phi: LG \times L\g^* \rightarrow  L\g^* \times L\g^*$, is given by 
\begin{equation}\label{Eqt:Psi-momentmap}
	\Phi(\gamma,A)=(\gamma\cdot A, {\rm inv}(A)).
\end{equation}
And the 2-form $\varpi\in \Omega^2(L\g^*)$ is given by:
\begin{equation}\label{Eqt:varpi}
\varpi:=\frac{1}{2}\int_0^1({\rm Hol}_s^*\bar{\theta}, \frac{d}{ds}{\rm Hol}_s^*\bar{\theta})ds.
\end{equation}
\end{thm}

\begin{remark}
Equation \eqref{eq:form-rel} above is first stated in \cite{Xu} without a proof. We carry out a detailed proof here. The map $f$ can be interpreted as a Morita morphism of Lie groupoids. Further, $\Phi$ serve as the moment map.  The equation above can be interpreted as the equivalence of a loop group-Hamiltonian space and a quasi-Hamiltonian space, see \cite{A-M-M}, where the form $\varpi$ is also introduced.
\end{remark}

We need an $LG\times LG$-action on $LG\times L\g^*$.

\begin{lemma}\label{prop:LHH-act}
There is an $LG\times LG$-action on $LG\times L\mathfrak{g}^*$, given by
\begin{equation}\label{eq:LHH_act_on_grpd}
		(g_1,g_2)\cdot (\gamma,A) = (g_1 \gamma (I^*g_2)^{-1}, (I^*g_2)\cdot A),
	\end{equation}
    where $g_1,g_2,\gamma \in LG$ and $A\in L\g^*.$ Indeed, this action is induced from the gauge action.
\end{lemma}

\begin{proof}
Following \cite{Xu}, we first recall a diffeomorphism
			\begin{align*}
			F:LG \times L\g^* & \rightarrow D_{\mu_D}\times_{\rm{Hol}} (L\g^* \times L\g^*),\\
			(\gamma,A) &\mapsto \big((\gamma(0),{\rm Hol}(A)),(\gamma \cdot A, {\rm inv}(A))\big).
		\end{align*}

Following the proof strategy of Theorem 8.3 in \cite{A-M-M},  we consider an action of $LG \times LG$ on $D_{\mu_D}\times_{\rm{Hol}}(L\g^*\times L\g^*)$ induced by the gauge action, which is given by 
	\begin{equation*}
		\begin{aligned}
			&(g_1,g_2)\cdot \big((u_1,u_2),(A_1,A_2)\big) \\= &\big((g_1(0)u_1g_2(0),\Ad_{g_2(0)}u_2),(g_1\cdot A_1, g_2 \cdot A_2)\big).   
		\end{aligned}
	\end{equation*}
	For any element $(\gamma,A) \in LG \times L\g^*$, according to Lemma \ref{lem:inversion-action}, we have
		\begin{align*}
			&(g_1,g_2)\cdot F(\gamma,A) \\
			=& (g_1,g_2)\cdot  \big((\gamma(0),{\rm Hol}(A)),(\gamma \cdot A, {\rm inv}(A))\big)\\
			=& \big((u_1,u_2),((g_1\gamma)\cdot A, g_2 \cdot {\rm inv}(A))\big)\\
			=& \big((u_1,u_2),((g_1\gamma(I^*g_2)^{-1})\cdot ((I^*g_2)\cdot A), {\rm inv}((I^*g_2)\cdot A))\big)\\
			=& F(g_1 \gamma (I^*g_2)^{-1}, (I^*g_2)\cdot A),
		\end{align*}
	where $u_1 = g_1(0)\gamma(0)g_2(0)^{-1}$ and $ u_2=g_2(0){\rm Hol}(A)g_2(0)^{-1}$. From this computation, we conclude that the 
$LG \times LG$-action on $LG\times L\g^*$ induced by the above fiber product action coincides with the one given in Equation \eqref{eq:LHH_act_on_grpd}.
\end{proof}

Let $\rho: L\g \times L\g \rightarrow \mathfrak{X}(LG \times L\g^*)$ be the infinitesimal action associated to the Hamiltonian $(LG\times LG)$-action defined by Equation \eqref{eq:LHH_act_on_grpd}. For any $(u,v)\in L\g\times L\g$, by Definition 8.2 of \cite{A-M-M}, we have the moment map condition 
\begin{equation}\label{eq:moment-cond}
	i_{\rho(u,v)}\sigma = d\int_{\mathbb{S}^1} (\Phi, (u,v)).
\end{equation}

\begin{lemma}\label{lem:1st-comp}
	For any $v\in  L\g$, we have  
    \begin{equation}\label{Eqt:omega=sigam}
    i_{\rho(v,0)}\omega_{LG\times L\g^*} = i_{\rho(v,0)}\sigma.\end{equation}
\end{lemma}

\begin{proof}
	For any $(\gamma,A) \in LG\times L\g^*$,  
	by Equation \eqref{eq:LHH_act_on_grpd}, we can derive the following:
	\begin{align*}
		\rho(v,0)_{(\gamma,A)}&=\frac{d}{dt}|_{t=0}(\exp(-tv),1)\cdot (\gamma,A)\\
		&=(\frac{d}{dt}|_{t=0} \gamma\gamma^{-1}\exp(-tv)\gamma,0)\\
		&=(-(L_\gamma)_*\Ad_{\gamma^{-1}}v,0).
	\end{align*}
	
	For the left hand side of Equation \eqref{Eqt:omega=sigam}, for $(u,B) \in L\g \times L\g^* \cong T_{(\gamma,A)}LG\times L\g^*$, using Equation \eqref{eq:O_Lh-def}  we compute as follows:
	\begin{equation*}
		\begin{aligned}
			&\omega_{LG\times L\g^*}(\rho(v,0),(u,B))_{(\gamma,A)}\\
			=& \omega_{LG\times L\g^*}((-\Ad_{\gamma^{-1}}v,0),(u,B))_{(\gamma,A)}\\
			=& \int_{\mathbb{S}^1} (\Ad_{\gamma^{-1}}v,B)+\int_{\mathbb{S}^1}(A,[\Ad_{\gamma^{-1}}v,u])+\lambda(\Ad_{\gamma^{-1}}v,u).
		\end{aligned}
	\end{equation*}
	And for the right hand side of Equation \eqref{Eqt:omega=sigam}, we have by Equation \eqref{eq:moment-cond} that
	\begin{equation*}
		\begin{aligned}
			i_{\rho(v,0)}\sigma
			=d\int_{\mathbb{S}^1}(\Phi,(v,0)).
		\end{aligned}
	\end{equation*}
    
	Suppose that $\beta(s,t)$ is a smooth family of curves in $G$ such that $\beta(s,0)\equiv e$ and that $\frac{\partial}{\partial t}|_{t=0} \beta(s,t)= u(s)$, then we may use the curve $(\gamma(s)\beta(s,t), A + tB)$ to represent the tangent vector $(u,B)$. 
	
	Now by pairing $(u,B)$ with $i_{\rho(v,0)}\sigma$, we get that
	\begin{equation*}
		\begin{aligned}
			&\sigma(\rho(v,0),(u,B))_{(\gamma,A)}\\
			=& \int_{\mathbb{S}^1} \frac{\partial}{\partial t}|_{t=0}\big(\Phi(\gamma\beta, A+tB),(v,0)\big)\\
			=& \int_{\mathbb{S}^1} \frac{\partial}{\partial t}|_{t=0}\big(\Ad_{\gamma\beta}(A + tB)+ \frac{\partial(\gamma\beta)}{\partial s}\beta^{-1}\gamma^{-1}ds,v\big) \qquad \mbox{by Equation \eqref{Eqt:Psi-momentmap}}\\
			=&  \int_{\mathbb{S}^1} (\Ad_{\gamma}[u,A],v) +\int_{\mathbb{S}^1} (\Ad_\gamma B,v)+ \int_{\mathbb{S}^1}\big(\frac{\partial}{\partial t}|_{t=0}(\frac{\partial(\gamma\beta)}{\partial s}\beta^{-1}\gamma^{-1}),v\big)ds.
		\end{aligned}
	\end{equation*}
	For the last term, we have that
	\begin{equation*}
		\begin{aligned}
			& \frac{\partial}{\partial t}|_{t=0}(\frac{\partial(\gamma\beta)}{\partial s}\beta^{-1}\gamma^{-1})\\
			=&\frac{\partial}{\partial t}|_{t=0}(\gamma \frac{\partial\beta}{\partial s}\beta^{-1}\gamma^{-1})\\
			=& \Ad_\gamma\frac{\partial}{\partial t}|_{t=0}(\frac{\partial\beta}{\partial s}\beta^{-1}).
		\end{aligned}
	\end{equation*}
	And finally we recall the identity (see e.g. \cites{MR1738431,crainic-fernandes})
	\begin{equation*}
		\frac{\partial}{\partial t}((\frac{\partial}{\partial s}\beta)\beta^{-1}) -  \frac{\partial}{\partial s}((\frac{\partial}{\partial t}\beta)\beta^{-1}) = [(\frac{\partial}{\partial s}\beta)\beta^{-1}, (\frac{\partial}{\partial t}\beta)\beta^{-1}]. 
	\end{equation*}
	Since $\beta(s,0)\equiv e$, we have that $\frac{\partial\beta}{\partial s}|_{(s,0)}\equiv 0$. Therefore,
	\begin{equation*}
		\frac{\partial}{\partial t}|_{t=0}((\frac{\partial}{\partial s}\beta)\beta^{-1}) = \frac{\partial}{\partial s}((\frac{\partial}{\partial t}|_{t=0}\beta)\beta(s,0)^{-1}) = \frac{\partial}{\partial s}u.
	\end{equation*}
	Since the inner product $(\cdot,\cdot)$ is $\Ad$-invariant, it follows that
	\begin{equation*}
		\begin{aligned}
			&\sigma(\rho(v,0),(u,B))_{(\gamma,A)}\\
			=&\int_{\mathbb{S}^1} (\Ad_{\gamma}[u,A],v) +\int_{\mathbb{S}^1} (\Ad_\gamma B,v)+ \int_{\mathbb{S}^1}(\Ad_{\gamma} \frac{\partial}{\partial s}u, v) ds\\
			=& \omega_{LG\times L\g^*} (\rho(v,0),(u,B))_{(\gamma,A)}.
		\end{aligned}
	\end{equation*}
	
\end{proof}

To proceed, we need a technical lemma.
\begin{lemma}\label{lem: inv_minus}
	We have that ${\rm inv}^*\varpi = -\varpi$.
\end{lemma}

\begin{proof}
	Let $V$ be a tangent vector at $A \in L\g^*$, which is in particular an element in $L\g^*$. Suppose $A_t$ is a curve in $L\g^*$ representing $V$. 
	According to Lemma  \ref{lem: invA} and the right invariance of $\bar{\theta}$, we have
	\begin{equation*}
		\begin{aligned}
			 ({\rm inv}^*{\rm Hol}_s^*\bar{\theta})(V)
			&= \bar{\theta}( \frac{d}{dt}|_{t=0} {\rm Hol}_s ({\rm inv}(A_t)))   \\
            &=\bar{\theta}( \frac{d}{dt}|_{t=0}  ({\rm Hol}_1(A_t)^{-1}{\rm Hol}_{1-s}(A_t)  )\\
			&= \bar{\theta}(\frac{d}{dt}|_{t=0} {\rm Hol}_1(A_t)^{-1} {\rm Hol}_{1-s}(A ))\\
			&\quad +\bar{\theta}( \frac{d}{dt}|_{t=0} {\rm Hol}_1(A)^{-1} {\rm Hol}_{1-s}(A_t ) )
		\end{aligned}
	\end{equation*}
For the first term, we use that $$(\frac{d}{dt}|_{t=0}{\rm Hol}_1(A_t)^{-1}){\rm Hol}_1(A) = - {\rm Hol}_1(A)^{-1}\frac{d}{dt}|_{t=0}{\rm Hol}_1(A_t).$$
By the right-invariant of $\bar{\theta}$, we have
\begin{equation*}
\begin{aligned}
    &\bar{\theta}(\frac{d}{dt}|_{t=0} {\rm Hol}_1(A_t)^{-1} {\rm Hol}_{1-s}(A ))\\
    =&-\bar{\theta} ({\rm Hol}_1(A)^{-1}\frac{d}{dt}|_{t=0}{\rm Hol}_1(A_t) )\\
    =& -\bar{\theta}( {\rm Hol}_1(A)^{-1}(\frac{d}{dt}|_{t=0}{\rm Hol}_1(A_t){\rm Hol}_1(A)^{-1}){\rm Hol}_1(A) )\\=& -\Ad_{{\rm Hol}(A)^{-1}} \overline{\theta}({\rm Hol}_* V).
\end{aligned}
\end{equation*}
For the second term, we compute that
\begin{equation*}
\begin{aligned}
&\bar{\theta}( \frac{d}{dt}|_{t=0} {\rm Hol}_1(A)^{-1} {\rm Hol}_{1-s}(A_t ) )\\
=&\bar{\theta}(\frac{d}{dt}|_{t=0} {\rm Hol}_1(A)^{-1} {\rm Hol}_{1-s}(A_t ){\rm Hol}_{1-s}(A )^{-1}{\rm Hol}_1(A) )\\
=& {\rm Ad}_{{\rm Hol_1(A)^{-1}}}\overline{\theta}({\rm Hol}_{1-s}^*V).   
\end{aligned}
\end{equation*}
    
	Hence we have
	$${\rm inv}^*{\rm Hol}_s^*\bar{\theta}= -\Ad_{{\rm Hol}(A)^{-1}} {\rm Hol}^*\bar{\theta}+\Ad_{{\rm Hol}(A)^{-1}}{\rm Hol}_{1-s}^*\bar{\theta}.$$
	
	Note that ${\rm Hol}_0^*\bar{\theta}=0$ and $({\rm Hol}^*\bar{\theta},{\rm Hol}^*\bar{\theta}) = 0. $ By the fundamental theorem of calculus, we have that $$\int_0^1({\rm Hol}^*\bar{\theta}, \frac{\partial}{\partial s}{\rm Hol}^*\bar{\theta}) = 0.$$
    
	Then it follows from the definition of $\varpi$ and $\Ad$-invariance of $(\cdot, \cdot)$ that
	\begin{equation*}
		\begin{aligned}
			{\rm inv}^*\varpi
			&=\frac{1}{2}\int_0^1({\rm inv}^*{\rm Hol}_s^*\bar{\theta}, \frac{\partial}{\partial s}{\rm inv}^*{\rm Hol}_s^*\bar{\theta})ds\\
			&=\frac{1}{2} \int_0^1(-{\rm Hol}^*\bar{\theta}+{\rm Hol}_{1-s}^*\bar{\theta}, \frac{\partial}{\partial s}{\rm Hol}_{1-s}^*\bar{\theta})ds\\
			&=\frac{1}{2} \int_0^1({\rm Hol}_{1-s}^*\bar{\theta}, \frac{\partial}{\partial s}{\rm Hol}_{1-s}^*\bar{\theta})ds\\
			&= -\varpi.
		\end{aligned}
	\end{equation*}
	
\end{proof}

In order to show Theorem \ref{thm:for-equiv}, we still need to check the following.

\begin{lemma}\label{lem:2nd-comp}
	For any $W_1,W_2 \in L\g^*$, we have that
	$$
	\omega_{LG\times L\g^*}((0,W_1),(0,W_2))=\sigma((0,W_1),(0,W_2))=0.
	$$
\end{lemma}

\begin{proof}
	It is clear that $\omega_{LG\times L\g^*}((0,W_1),(0,W_2))=0.$ Also, we compute by Equation \eqref{eq:O_D-def} that $$f^* \omega_D\big((0,W_1),(0,W_2)\big)_{(h,\eta)} = \omega_D\big((0,{\rm Hol}_*W_1),(0,{\rm Hol}_*W_2)\big)_{(h(0),Hol(\eta))}=0.$$

	Moreover, denote by $\varrho(h): L\g^* \rightarrow L\g^*$ to be the action map for Equation \eqref{eq: LH_on_(Lh)*}, then we have by Lemma \ref{lem: inv_minus} that
	\begin{equation*}
		\begin{aligned}
			&\Phi^*(\varpi,\varpi)((0,W_1),(0,W_2))_{(h,\eta)}\\
			=& {\rm inv}^*\varpi(W_1,W_2) + \varrho(h)^*\varpi(W_1,W_2)\\
			=& -\varpi(W_1,W_2) + \varrho(h)^*\varpi(W_1,W_2)
		\end{aligned}
	\end{equation*}
	By Lemma A.1 of \cite{A-M-M}, we have
	\begin{equation}
		\varrho(h)^* {\rm Hol}^*_s \bar{\theta}=\Ad_{h(0)} {\rm Hol}^*_s \bar{\theta}.
	\end{equation}
	By the $\Ad$-invariance of $(\cdot,\cdot)$ we have that
	\begin{equation*}
		\begin{aligned}
			\varrho(h)^*\varpi& = \int_0^1(\varrho(h)^*{\rm Hol}_s^*\bar{\theta}, \varrho(h)^*\frac{d}{ds}{\rm Hol}_s^*\bar{\theta})ds \\
			& = \int_0^1(\Ad_{h(0)}{\rm Hol}_s^*\bar{\theta}, \Ad_{h(0)}\frac{d}{ds}{\rm Hol}_s^*\bar{\theta})ds = \varpi.
		\end{aligned}
	\end{equation*}
	
	Therefore we get that $\Phi^*(\varpi,\varpi)=0$.
\end{proof}

It is clear that Theorem \ref{thm:for-equiv} follows from Lemma \ref{lem:1st-comp} and Lemma \ref{lem:2nd-comp}.

\subsection{Identification of cohomology classes}
Now we derive some useful consequences of Theorem \ref{thm:for-equiv}. 
\begin{propn}\label{prop: Morita-grpd}\cite{Xu}
	The projection map 
	\begin{equation*}
		\begin{aligned}
			f:LG \times L\g^* &\rightarrow G\times G,\\
			(\gamma,A)&\mapsto (\gamma(0), \rm{Hol}(A)),
		\end{aligned}
	\end{equation*}
	is a Morita morphism of Lie groupoids.
\end{propn}

We first recall the following technical lemma.
\begin{lemma}[\cite{A-M-M}]\label{lem:AMM}
For the 2-form $\varpi\in\Omega^2(L\g^*)$ defined by Equation \eqref{Eqt:varpi}, 	we have that $$d\varpi = - \rm{Hol}^*\Omega,$$ where ${\rm Hol}\colon L\g^*\to G$ is the holonomy map and  $\Omega\in \Omega^3(G)$ is the invariant 3-form.    
\end{lemma}

Combining Theorem \ref{thm:for-equiv} and Lemma \ref{lem:AMM}, we get the following result, which states that the Morita morphism $f$ sends the groupoid cohomology class $[\omega_D +\Omega ]$ to $[\omega_{LG\times L\g^*}]$. The formula is first stated in \cite{B-X-Z} without a proof.

\begin{propn}\label{prop:formula-Xu}
	Let $\delta$ be the de Rham differential on the groupoid $LG \times L\g^*$. We have that
	\begin{equation}\label{eq:formula-Xu}
		\delta\varpi 
		=\omega_{LG\times L\g^*}-f^*(\omega_D+\Omega).
	\end{equation}
\end{propn}

\begin{proof}
First, recall that the Morita morphism $f\colon LG\times L\g^*\to G\times G$ is compatible with the base manifold map ${\rm Hol}\colon L\g^*\to G$. Then it follows that $$f^*(\omega_D+\Omega)=f^*\omega_D+{\rm Hol}^*\Omega.$$
Next, by the definition of the groupoid de Rham differential $\delta$, we have that $$\delta\varpi = \mathbf{s}^*\varpi-\mathbf{t}^*\varpi+ d\varpi.$$ We may write the moment map $\Phi$ as $\Phi = (\mathbf{t}, {\rm inv}\circ\mathbf{s})$. Then by Lemma \ref{lem: inv_minus} we compute that
	$$
	\Phi^*(\varpi,\varpi) = \mathbf{t}^*\varpi + (\rm{inv}\circ\mathbf{s})^*\varpi = \mathbf{t}^*\varpi - \mathbf{s}^*\varpi.
	$$
By Theorem \ref{thm:for-equiv}, the symplectic form $\omega_{LG\times L\g^*}$ satisfies the identity:
	$f^*\omega_D-\Phi^*(\varpi,\varpi) = \omega_{LG\times L\g^*} $. Therefore, by the above equations and Lemma \ref{lem:AMM}, 
		\begin{align*}
			\delta \varpi = -\Phi^*(\varpi,\varpi)+d\varpi &= \omega_{LG\times L\g^*}-f^*\omega_D - \rm{Hol}^*\Omega\\
            &=\omega_{LG\times L\g^*}-f^*(\omega_D+\Omega).
		\end{align*}
\end{proof}

\begin{remark}
The above result finds applications in the study of Morita equivalence of quasi-symplectic groupoids \cite{Xu,A-M_24}, prequantization \cite{Krepski}, and the unversal characteristic class of 2-group bundles, see Example 4.28 of \cite{mathieupaper}.
\end{remark}

\begin{remark}
When $G$ is compact, connected, and simply connected, the above result admit an elegant interpretation through the lens of moduli space of flat connections, see \cite{A-M-M, meinrenken1996symplecticproof}. See \cite{A-M_08} for the framework of possible extension to general cases. 
\end{remark}

\section{An $\mathbb{S}^1$-gerbe over the stack $[G/G]$}
In this section, we discuss the stack-theoretic interpretation of our previous results, connecting the constructed symplectic groupoids and their extensions to equivariant cohomology classes. We begin by reviewing the definitions of stacks and gerbes, and then provide detailed explanations of how the $\mathbb{S}^1$-extension of the groupoid $LG\times L\g^*\rightrightarrows L\g^*$ leads to the construction of an $\mathbb{S}^1$-gerbe over the differentiable stack  $[G/G]$.

\subsection{$\mathbb{S}^1$-gerbes and Dixmier-Douady classes}\label{Sec:5.1}

The notion of stacks is usually defined in an abstract way. For our context, we focus on the differentiable stacks, introduced by Behrend and Xu in \cite{stacks_B-X}.  Indeed,  \textbf{differentiable stacks} are Lie groupoids up to Morita equivalence. For any Lie groupoid $\mathcal{G}$, we use $[\mathcal{G}]$ to denote the corresponding differentiable stack. Due to conventional usage in the literature, we adopt the notation $[G/G]$ for the differentiable stack $[G\times G]$. For any sheaf $\mathcal{F}$ and any Lie groupoid $\mathcal{G}$, the sheaf cohomology of Lie groupoids ${\rm H}^\bullet(\mathcal{G},\mathcal{F})$ is invariant under Morita equivalence. Thus, we take the cohomology group ${\rm H}^\bullet(\mathcal{G},\mathcal{F})$ as the \textbf{sheaf cohomology of differentiable stack [$\mathcal{G}$]}, denoted as ${\rm H}^\bullet([\mathcal{G}],\mathcal{F})$.

An \textbf{$\mathbb{S}^1$-central extension} of a Lie groupoid $\mathcal{G}\rightrightarrows M$ is a groupoid $\widetilde{\mathcal{G}}\rightrightarrows M$ fitting in a short exact sequence of Lie groupoids
\begin{equation*}
     \mathbb{S}^1\times M\xrightarrow[]{\widetilde{i}}\widetilde{\mathcal{G}} \xrightarrow[]{\widetilde{p}}\mathcal{G},
\end{equation*}
where $\widetilde{i}$,  $\widetilde{p}$ are  injective, surjective groupoid morphisms over $\operatorname{id}_M$, respectively, and for any $\widetilde{g}\in \widetilde{\mathcal{G}} , r \in \mathbb{S}^1$, we have that $\widetilde{i}(\mathbf{t}(\widetilde{g}),r)\cdot \widetilde{g} =\widetilde{g} \cdot  \widetilde{i}(\mathbf{s}(\widetilde{g}),r).$

According to \cite{stacks_B-X}, an \textbf{$\mathbb{S}^1$-gerbe} over a differentiable stack $[\mathcal{G}]$ can be viewed as an  $\mathbb{S}^1$-central extensions of Lie groupoids $\widetilde{\mathcal{G}}\rightarrow \mathcal{G}$, up to Morita morphism of central extensions.  We shall use $[\widetilde{\mathcal{G}}\rightarrow \mathcal{G}]$ to denote this $\mathbb{S}^1$-gerbe.

\begin{example}
Given a central extension of Lie groups $1\rightarrow \mathbb{S}^1 \rightarrow \widetilde{G} \rightarrow G \rightarrow 1$, we have that $\widetilde{G}\times \g^*\rightarrow G\times \g^*$ is an $\mathbb{S}^1$-central extension of Lie groupoids. Therefore, $[\widetilde{G}\times \g^*\rightarrow G\times \g^*]$ is an $\mathbb{S}^1$-gerbe over the stack $[G\times \g^*]$. 
\end{example}

Indeed, $\mathbb{S}^1$-central extensions give meanings to certain cohomology classes, due to the following deep result of Giraud.

\begin{thm}[\cite{Giraud}]
There is a one-to-one correspondence between the isomorphism classes of $\mathbb{S}^1$-gerbes over a differentiable stack $[\mathcal{G}]$, and the cohomology group ${\rm H}^2([\mathcal{G}], \mathbb{S}^1)$.
\end{thm}

In particular, any central extension of Lie groupoids $\widetilde{\mathcal{G}}\rightarrow \mathcal{G}$ defines a class $[\widetilde{\mathcal{G}}\rightarrow \mathcal{G}] \in {\rm H}^2([\mathcal{G}], \mathbb{S}^1)$. Such classes can be related to higher degree cohomology. Consider the exponential sequence of sheaves $0\rightarrow \mathbb{Z} \rightarrow \Omega^0 \rightarrow \mathbb{S}^1 \rightarrow 0$. This sequence gives rise to a long exact sequence in cohomology, yielding a connecting morphism ${\rm H}^2([{\mathcal{G}}],\mathbb{S}^1) \rightarrow {\rm H}^3([{\mathcal{G}}],\mathbb{Z})$. 

\begin{defn}
Let $\widetilde{\mathcal{G}}\rightarrow \mathcal{G}$ be an $\mathbb{S}^1$-central extension of Lie groupoids. Under the connecting morphism ${\rm H}^2([{\mathcal{G}}],\mathbb{S}^1) \rightarrow {\rm H}^3([{\mathcal{G}}],\mathbb{Z})$, the image of the $\mathbb{S}^1$-gerbe $[\widetilde{\mathcal{G}}\rightarrow \mathcal{G}]$ in ${\rm H}^3([{\mathcal{G}}],\mathbb{Z})$ is called the \textbf{Dixmier-Douady class} of $[\widetilde{\mathcal{G}}\rightarrow \mathcal{G}]$.   
\end{defn}

\begin{remark}\label{Rem:AMM-gerbe}
Note that for a compact Lie group $G$, the AMM groupoid $D=G \times G \rightrightarrows G$ is a proper Lie groupoid, as the map $(s,t): D \to G \times G$ is proper. This property is significant for the theory of gerbes, as it ensures the isomorphism ${\rm H}^2([G/G], \mathbb{S}^1) \cong {\rm H}^3([G/G], \mathbb{Z})$, which means the Dixmier-Douady class truly classifies the $S^1$-gerbes over $D$ (up to Morita equivalence).
\end{remark}

Dixmier-Douady classes serve as higher analogues of Chern classes, representing a topological invariant that characterizes the extension.
Similar to how Chern classes are computed from curvature data, Dixmier-Douady classes for groupoid extensions can be related to curvature-like 3-forms. This essential connection is established by Behrend and Xu \cite{stacks_B-X}. For an $\mathbb{S}^1$-central extension of Lie groupoids $\widetilde{p}\colon\widetilde{\mathcal{G}} \rightarrow \mathcal{G}$, a \textbf{pseudo-connection} is defined as a pair $(\vartheta, B)$, where $\vartheta \in \Omega^1(\widetilde{\mathcal{G}})$ is a connection 1-form for the associated principal $\mathbb{S}^1$-bundle, and $B \in \Omega^2(M)$ is a 2-form on the base manifold. The \textbf{pseudo-curvature} of this pseudo-connection is a specific 3-form $\eta + \omega + \Omega$, where $\eta \in \Omega^1(\mathcal{G}_2)$, $\omega\in \Omega^2(\mathcal{G})$, and $\Omega \in \Omega^3(M)$ are derived from $\delta(\vartheta + B) = p^*(\eta + \omega + \Omega)\in C_{\rm dR}^3(\widetilde{\mathcal{G}})$. Behrend and Xu proved the following fundamental result relating pseudo-curvature to the Dixmier-Douady class:
\begin{thm}[\cite{stacks_B-X}]\label{thm: Behrend-Xu}
Given any pseudo-connection $(\theta, B)$ and its pseudo-curvature $\eta + \omega + \Omega$, the corresponding de Rham class $[\eta + \omega + \Omega] \in {\rm H}_{\rm dR}^3([\mathcal{G}])$ is independent of the choice of pseudo-connection. 

Under the canonical homomorphism ${\rm H}^3([\mathcal{G}],\mathbb{Z})\rightarrow  {\rm H}^3([\mathcal{G}], \mathbb{R}) = {\rm H}_{\rm dR}^3([\mathcal{G}])$, the Dixmier-Douady class of $[\widetilde{\mathcal{G}}\rightarrow \mathcal{G}]$ maps to $[\eta + \omega + \Omega]$. In particular, the class $[\eta + \omega + \Omega] \in {\rm H}_{\rm dR}^3([\mathcal{G}])$ is an integer class.
\end{thm}

We may apply this framework to our construction. As established in Section \ref{Sec:Second-construction}, given an $\mathbb{S}^1$-central extension $\widetilde{G}\rightarrow G$, our constructed symplectic groupoid $(\Gamma,\omega_\Gamma)$ possesses a natural $\mathbb{S}^1$-central extension $\widetilde{G}\times \g^*\rightarrow \Gamma$, which can be viewed as a principal $\mathbb{S}^1$-bundle. The symplectic form $\omega_{\Gamma}$ itself naturally arises from the symplectic reduction of the minus canonical form $-\widetilde{\varphi}^*\widetilde{\omega}_{\rm can}$ on $\widetilde{\Gamma}:=\widetilde{G}\times\widetilde{\g}^*$. Let $\widetilde{j}: \widetilde{G}\times \g^*\rightarrow \widetilde{\Gamma}$ be the inclusion map and $\theta_L\in \Omega^1(\widetilde{\Gamma})$ be the canonical Liouville form. According to Proposition \ref{prop:rel_OG&OC}, it is clear that $$d(\widetilde{j}^*\widetilde{\varphi}^*\theta_L) = -\widetilde{j}^*\widetilde{\varphi}^*\widetilde{\omega}_{\rm can} = \widetilde{p}^*\omega_{\Gamma}.$$ Note that $\widetilde{j}^*\widetilde{\varphi}^*\theta_L$ is a connection 1-form for the principal $\mathbb{S}^1$-bundle $\widetilde{G}\times \g^*\rightarrow \Gamma$. Consequently, $\widetilde{j}^*\widetilde{\varphi}^*\theta_L$ is a pseudo-connection for our central extension, whose pseudo-curvature is precisely $\omega_{\Gamma}$.
Therefore, we conclude:
\begin{propn}\label{cor: DD-class-omega}
The Dixmier-Douady class of the $\mathbb{S}^1$-gerbe $[\widetilde{G}\times \g^*\rightarrow \Gamma]$ maps to the de Rham class $[\omega_{\Gamma}] \in {\rm H}_{\rm dR}^3([\Gamma])$ under the canonical homomorphism ${\rm H}^3([\Gamma], \mathbb{Z}) \rightarrow {\rm H}_{\rm dR}^3([\Gamma])$.
\end{propn}

\subsection{Proof of the Main Theorem}\label{sec: ThmA}
Now we return to the main theme of this paper, namely, constructing the basic $\mathbb{S}^1$-gerbe over $[G/G]$. We shall construct explicitly cental extensions of Lie groupoids whose Dixmier-Douady classes coincides with the AMM class. Let $G$ be a compact, connected Lie group, with Lie algebra $\mathfrak{g}$.

The differentiable stack $[G/G]$ can be presented by both the action groupoids $G\times G$ and $LG \times L\mathfrak{g}^*.$ We shall pick $LG \times L\mathfrak{g}^*$ as the base for the central extensions, and the reason is as follows. The existence and classification of $\mathbb{S}^1$-central extensions for a Lie groupoid $\mathcal{G}_1\rightrightarrows \mathcal{G}_0$ are governed by a long exact sequence (Proposition 4.7 of \cite{stacks_B-X}):
\begin{align*}
{\rm H}^1(\mathcal{G}_0, \mathbb{S}^1)& \to \{ \mathbb{S}^1\text{-central~extensions ~of ~}\mathcal{G}_1\rightrightarrows \mathcal{G}_0\}\\
&\to {\rm H}^2([\mathcal{G}_1], \mathbb{S}^1)\to{\rm H}^2(\mathcal{G}_0, \mathbb{S}^1).
\end{align*}
In our case, we have $\mathcal{G}_0 = L\mathfrak{g}^*$. This choice is strategic because $L\mathfrak{g}^*$ is a contractible (infinite-dimensional) vector space, meaning its singular cohomology ${\rm H}^k(L\mathfrak{g}^*, \mathbb{Z})$ vanishes for $k>0$. This property simplifies the long exact sequence, ensuring that  any $\mathbb{S}^1$-gerbe over $[G/G]$ can be represented by a corresponding $\mathbb{S}^1$-central extension of a Lie groupoid based at $L\mathfrak{g}^*$.



The loop Lie algebra cocycle in Equation \eqref{eq:loopcocycle} gives rise to a Lie algebra structure on the extended loop algebra $\widetilde{L\g}: = L\mathfrak{g} \oplus \mathbb{R},$
\begin{equation*}
	[(v_1,r_1),(v_2,r_2)]=([v_1,v_2], \lambda(v_1,v_2)),
\end{equation*} 
where $(v_i,r_i)\in \widetilde{L\g}.$
Then it gives a central extension of Lie algebras:
\begin{equation*}
	0\rightarrow \mathbb{R}\rightarrow \widetilde{L\g} \rightarrow L\g\rightarrow 0. 
\end{equation*}
 If the central extension above to an $\mathbb{S}^1$-central extension of Lie groups
\begin{equation}\label{eq: Kac-Moody}
	1\rightarrow \mathbb{S}^1\rightarrow \widetilde{LG} \rightarrow LG\rightarrow 1,
\end{equation}
then this is called the Kac-Moody extension. From now on, we assume that $G$ is compact, connected, and the Kac-Moody extension for $LG$ exists (for example, if $G$ is simply connected \cite{P-Segal}). 

 By the discussions in Section \ref{Sec:Second-construction}, the space $\widetilde{LG}\times \widetilde{L\g^*}$ with the canonical 2-form, constitute a Hamiltonian $\mathbb{S}^1$-space, with moment map $\widetilde{\mu}$. Here, following \cite{meinrenken1996symplecticproof} we take the convention $\widetilde{L\g^*} = L\g^*\oplus \mathbb{R}$. It is then clear that $\widetilde{\mu}^{-1}(1) = \widetilde{LG}\times L\g^*$ and $\widetilde{\mu}^{-1}(1)/\mathbb{S}^1 = LG\times L\g^*$. We thus obtain an $\mathbb{S}^1$-central extension of Lie groupoids
    $$   \mathbb{S}^1 \times L\g^*\xrightarrow[]{\widetilde{i}}\widetilde{LG}\times L\g^*\xrightarrow[]{\widetilde{p}}LG\times L\g^*  ,$$
which represents an $\mathbb{S}^1$-gerbe $[\widetilde{LG}\times L\g^*\rightarrow LG\times L\g^*]$ over the stack $[LG\times L\g^*] = [G/G]$. 

To study the Dixmier-Douady class of $[\widetilde{LG}\times L\g^*\rightarrow LG\times L\g^*]$, we may again take the Liouville 1-form on $\widetilde{LG}\times \widetilde{L\g^*}$, which pullback to a connection 1-form on $\widetilde{LG}\times L\g^*$. By the discussion in Proposition \ref{cor: DD-class-omega}, we have that $\omega_{LG\times L\g^*}$ is the pseudo-curvature. We now conclude by Theorem \ref{thm: Behrend-Xu} and Proposition \ref{prop:formula-Xu} that

\begin{thm}
The $\mathbb{S}^1$-central extension $\widetilde{LG}\times L\g^*\rightarrow LG\times L\g^*$ defines an $\mathbb{S}^1$-gerbe $[\widetilde{LG}\times L\g^*\rightarrow LG\times L\g^*]$ over the stack $[G/G]$. Its Dixmier-Douady class has image $[\omega_{LG\times L\g^*}] = [\omega_D+\Omega] \in {\rm H}^3_{\rm dR}([G/G])$, under the canonical morphism ${\rm H}^3([G/G], \mathbb{Z}) \rightarrow {\rm H}_{\rm dR}^3([G/G])$.
\end{thm}

Finally, we explain why the $\mathbb{S}^1$-gerbe constructed here is the basic gerbe. For any compact, simply connected simple Lie group $G$, it is well-known that the degree-3 integral equivariant cohomology group ${\rm H}^3_G(G, \mathbb{Z})$ is canonically isomorphic to $\mathbb{Z}$. Under the inclusion ${\rm H}^3_G(G, \mathbb{Z})\subset {\rm H}^3_G(G)$, the image of the canonical generator of $\mathbb{Z}$ is represented by the equivariant 3-form $\chi_G: \mathfrak{g} \rightarrow \Omega^\bullet (G)$ in the Cartan model, given by $\chi_G(\xi) = \Omega - \frac{1}{2}(\theta + \overline{\theta}, \xi )$. An $\mathbb{S}^1$-gerbe is called basic \cite{Meinrenken}, if its Dixmier-Douady class is represented by $[\chi_G]$. In \cite{Bursztyn2003IntegrationOT}, it is explicitly computed that, under the canonical isomorphism ${\rm H}^3_G(G) \cong {\rm H}_{\rm dR}^3(G\times G)$, the class $[\chi_G]$ is identified with the AMM class. Since the Dixmier-Douady class of the $\mathbb{S}^1$-gerbe from our construction coincides with the AMM class, we conclude that it generates the basic $\mathbb{S}^1$-gerbe.

\vspace{20pt}  

\noindent \textbf{Acknowledgments.}
We are deeply grateful to Professor Ping Xu for suggesting this fascinating
research problem and for many insightful discussions. We thank Zhuo Chen,
Luen-Chau Li, Zhangju Liu, Mathieu Sti\'enon, Marco Zambon, and Bin Zhang for
inspiring discussions and useful feedback. This work was supported by the
Research Institute for Mathematical Sciences, an International Joint Usage/
Research Center located in Kyoto University.

\end{document}